\documentclass[reqno,11pt]{amsart}

\pdfoutput=1

\usepackage{amsfonts}
\usepackage{amsfonts}
\usepackage{mathrsfs}
\usepackage{graphicx}
\usepackage{amsfonts}
\usepackage[dvipsnames,usenames]{color}
\usepackage{amsmath, amsthm, amsfonts, amssymb}

\newtheorem{thm}{Theorem}[section]

\newtheorem{lem}[thm]{Lemma}
\newtheorem{prop}[thm]{Proposition}
\newtheorem{example}[thm]{Example}

\newtheorem{question}[thm]{Question}
\theoremstyle{definition}
\numberwithin{equation}{section}
\theoremstyle{remark} \hsize=7.5truein \vsize=8.6truein

\newcommand{\R}{\ensuremath{\mathbb{R}}}

\title{Traffic Analysis in Random Delaunay Tessellations and Other Graphs}
\author{John D. Hobby \and Gabriel H. Tucci}

\begin{document}

\begin{abstract}
In this work we study the degree distribution, the maximum vertex and edge flow in non-uniform random Delaunay triangulations when geodesic routing is used. We also investigate the vertex and edge flow in Erd\"os-Renyi random graphs, geometric random graphs, expanders and random $k$--regular graphs. Moreover we show that adding a random matching to the original graph can considerably reduced the maximum vertex flow.
\end{abstract}

\maketitle

\section{Introduction and Motivation}

Random graphs constitute an important and active research area with numerous applications to geometry, percolation theory, information theory, queuing systems and communication networks, to mention a few.  They also provide analytical means to settle prototypical questions and conjectures that may be harder to resolve in specific circumstances. In this work we analyze the traffic congestion on random Delaunay triangulations with non-uniform density, Erd\"os--Renyi random graphs, random $k$--regular graphs and random geometric graphs. 

Given a set of points in Euclidean space, the Delaunay triangulation gives a canonical triangulation whose vertices are these points. These triangulations have many nice combinatorial and geometric properties that make them extremely useful. Moreover, they can also be constructed in Riemannian manifolds. However, they do not exist for arbitrary sets of points: certain density requirements are needed to ensure that the triangulation can accurately represent both the topology and geometry of the manifold. These triangulations are canonically determined by the set of points and 
have many of the properties that Delaunay triangulations have in the Euclidean space. 

The definition of Delaunay tessellations is the same in Riemannian geometry as it is in $\R^{n}$. They are defined as having the empty circumscribing sphere property: the minimal radius circumscribing sphere for any simplex contains no vertices of the tessellation in its interior. However, there are several possible problems with this definition. How do we know the circumscribing sphere is unique? For instance, a necessary requirement is that all simplices of our Delaunay tessellation and their neighbors are contained inside strongly convex balls. Otherwise, we would run into problems with the tessellation being well defined.  Since in this work we are not is interested in studying Delaunay triangulations per se but in understanding their traffic congestion and graph properties, we focus on the following construction. Let $M\subseteq\R^2$ be a either the open unit disk or the whole Euclidean plane with density $\rho(x,y)$, and let $\mathcal{P}$ be a Poisson point process of density $\rho$ on $M$. The distribution and density of these points are determined by the function $\rho$. Now let $\mathcal{T}$ be the Delaunay triangulation of $\mathcal{P}$ with respect to the Euclidean metric in $\R^2$. In this work we are mostly interested in the case where the density is rotationally symmetric (i.e. $\rho(x,y) = f(x^2+y^2)$) and where 
$$
\int_{M}{\rho(x,y)dxdy}=\infty.
$$
The last condition guarantees that the points in the set $\mathcal{P}$ are infinite with probability one. Our interest is in understanding how the graph's characteristics change as we change the density. For instance, let $\mathcal{T}$ be as before and choose $x_{0}$ an arbitrary vertex. Consider the sets 
$$
\mathcal{T}_{n}:=\{x\in \mathcal{T}\,:\,d(x_0,x)\leq n\}
$$
with the graph metric (hop metric). For each pair of nodes, consider a unit flow that travels through the minimum path between nodes 
so that the total flow in $\mathcal{T}_{n}$ is equal to $|\mathcal{T}_{n}|(|\mathcal{T}_n|-1)/2$.
If there is more than one minimum path for some pair of nodes $v,w$, the flow splits
among them in a locally equal manner:  If $v'$ is a vertex on a such a path, the flow
out of $v'$ is split equally among neighbors $\{x \,:\, d(v',x)=1$ and $d(x,w)=d(v',w)-1\}$,
where $d$ is the hop metric.

Given a node $v$ we define $T_{n}(v)$ as the total flow generated in $\mathcal{T}_{n}$ passing through $v$. In other words, $T_{n}(v)$ is the sum off all the geodesic paths in $\mathcal{T}_{n}$ which are carrying flow and contain the node $v$. Let $M_{v}$ be the maximum vertex flow
$$
M_{v}(n):=\max \Big\{T_{n}(v)\,:\, v\in \mathcal{T}_{n}\Big\}.
$$
It is easy to see that for any graph $|\mathcal{T}_{n}|-1\leq M_{n}\leq |\mathcal{T}_{n}|(|\mathcal{T}_n|-1)/2$. Analogously, we can define $M_{e}(n)$ as the maximum edge congestion. Note that the same question can be formulated by replacing the infinite graph $\mathcal{T}$ with a sequence of graphs $\{G_{n}\}_{n=1}^{\infty}$ such that $|G_{n}|\to\infty$.

One of our main motivations is to understand how $M_{v}$ is affected by changes in the function $\rho$. We also study how the density affects the degree distribution of the corresponding triangulation. Another topic is the maximum vertex flow in Erd\"os-Renyi random graphs, random geometric graphs and random $k$--regular graphs. 

It was observed in many complex networks, man-made or natural, that the typical distance between the nodes is surprisingly small. More formally, as a function of
the number of nodes $N$, the average distance between a node pair typically scales at or below $O(\log(N))$. Moreover, many of these complex networks, specially communication networks, have high congestion. More precisely, there exists a small number of nodes called the {\it core} where most of the traffic pass through. In this work, we present a possible solution to this problem. Given a graph $G=(V,E)$, consider the graph $H$ constructed by adding a random maximal matching to $G$. More precisely, we choose a pair of nodes at random from $G$ and add an edge between these two nodes if they are not already connected. Now, we remove these two nodes from the possible candidates and repeat this process until there is only one or no nodes remaining (depending on the parity of $|V|$). It is clear that the new graph $H$ has the same nodes as before, and it has at most $|V|/2$ extra edges. Moreover, we added at most an extra edge at every node.  We show that if the original graph satisfies certain hypotheses (exponential growth) then we can reduce the maximum vertex flow to $O(N\mathrm{poly}(\log(N)))$.

This paper is organized as follows. In Section \ref{Sconstruc}, we present the construction of the random Delaunay triangulations and we show a necessary condition to guarantee maximum vertex flow of the order $\Theta(N^2)$. We also show that for any planar graph, congestion cannot be smaller that $\Theta(N^{3/2})$ regardless of how the flow is routed. In Section \ref{analysis}, we analyze the maximum  flow and degree distribution for several examples. In Section \ref{rg}, we analyze maximum vertex and edge flow in random Erd\"os-Renyi random graphs, random geometric graphs and random $k$--regular graphs. We further present the main result on the congestion after a random matching. Finally, Section \ref{imp_alg} discusses some of the algorithms used, their complexity and how they were implemented.

\section{Construction}
\label{Sconstruc}
\par Consider $M=B(0,t)\subseteq \R^2$ the open ball of radius $t=1$ or $t=\infty$. Let $\rho$ be a rotationally symmetric density on $M$. Let $\mathcal{P}$ be a Poisson process with intensity $\rho$ on $M$, and let $\mathcal{T}$ be the Delaunay triangulation of the set of points in $\mathcal{P}$. In particular, introducing geodesic polar coordinates $r, \theta$ makes it apparent that for any Borel set $A\subseteq M$ the expected number of points in $A$ is
\begin{equation}
\int_{A}{\rho(r) r \,dr d\theta}.
\end{equation}
Therefore, if 
\begin{equation}\label{rho}
2\pi \int_{0}^{t}{\rho(r)r\,dr}=\infty
\end{equation}
then almost surely $\mathcal{T}$ induces and infinite graph (triangulation) on $M$.
\begin{example}\label{examp1}
If $t=1$ and 
$$
\rho(r) = \frac{4\lambda}{(1-r^2)^2},
$$
then $\rho$ corresponds to the constant density on the unit ball with the metric of the Hyperbolic space under the Poincar\'e disk representation with curvature $-\lambda^2$.
\end{example}

\begin{example}
If $t=\infty$ and 
$$
\rho(r) = 1
$$
then $\rho$ corresponds to the constant density on the Euclidean space.
\end{example}

\par For notational simplicity, we assume that $0\in\mathcal{P}$; otherwise we choose the closest point to the origin. Let $\{r_{n}\}_{n=1}^{\infty}$ be an increasing sequence of positive numbers such that $r_{n}<t$ and $\lim_{n\to\infty}{r_n}=t$. Consider $\mathcal{T}_{n}$ the graph induced by the finite set $\{x\in\mathcal{T}\,\,:\,\, d(0,x)\leq r_n\}$ where $d(x,y)$ for $x,y\in\mathcal{T}$ is the hop metric in the triangulation. In this Section, we are interested in analyzing the traffic behavior on the graph $\mathcal{T}_{n}$ as $n$ increases as we discussed in the introduction. For instance, if $G_N$ is $K_N$, the complete graph with $N$ vertices, then $M_{v}(N)=N-1$. It was proved, in \cite{Tucci1, Tucci2} that if $G$ is an infinite Gromov hyperbolic graph then $M_{v}(n)=\Theta(N^2)$.  (For more details in Gromov hyperbolic graphs, see \cite{Gromov}). In particular, trees have congestion of the order $N^2$.

\par One of our main questions is, what is the rate of growth of $M_v(n)$ and $M_e(n)$ as $n\to\infty$. In particular, under what conditions is
\begin{equation}
\lim_{n\to\infty}{\frac{M_v(n)}{N(n)^2}}>0?
\end{equation}

The next Theorem gives a sufficient condition, but first let us recall a result from Riemannian geometry. Given a rotationally symmetric Riemannian surface with metric $ds^2 = f(r)^2dr^2 + g(r)^2d\theta^2$, its curvature $K(r)$ is equal to
\begin{equation}\label{curvature}
K(r)=\frac{-1}{f(r)g(r)}\frac{\partial}{\partial r}\Bigg( \frac{g'(r)}{f(r)}\Bigg).
\end{equation}

\begin{thm} Assume $t$ is equal to $1$, and let $\rho$ be a differentiable, strictly increasing function that is strongly convex (i.e., $\ln\rho$ is convex) and satisfies condition (\ref{rho}). Then there exists $\alpha>0$ such that
\begin{equation}
\limsup_{n\to\infty}{\frac{M_v(n)}{N(n)^2}}=\alpha>0.
\end{equation}
\end{thm}

\begin{proof}
Let $M=B(0,1)$ be the Riemannian surface with metric  $ds^2 = f(r)^2dr^2 + r^2f(r)^2d\theta^2$. Then its area form $dA$ is equal to $dA=rf(r)^2drd\theta$. Hence, the Euclidean unit ball $B(0,1)$ with volume density $\rho(r)=f(r)^2$ corresponds to this metric. By applying equation (\ref{curvature}) to $M$ we see that 
$$
K(r) = \frac{-1}{rf(r)^2}\Bigg( \frac{f'(r)}{f(r)} + r\frac{\partial^2}{\partial r^2}\ln f(r)\Bigg).
$$
Since $f$ is increasing and strongly convex, we observe that $K(r)<0$ for all $0\leq r<1$. Let $\mathcal{P}$ be a Poisson process in $B(0,1)$ with intensity $\rho$. This is equivalent to a Poisson process in $M$ of constant unit intensity. Let $\mathcal{T}$ be its Delaunay triangulation, and let $\{r_{n}\}_{n=1}^{\infty}$ be an increasing sequence converging to 1. Denote by $\mathcal{T}_{n}$ the restriction of $\mathcal{T}$ to $B(0,r_n)$, and let $N(n)=|\mathcal{T}_{n}|$.

As before, we assume that $0\in\mathcal{P}$ (otherwise we choose the closest point to the origin). For each $0\leq\delta<1$, the volume $\mathrm{vol}(B(0,\delta))<\infty$ and $\mathrm{vol}(B(0,\delta)^{c})=\infty$. Therefore, we have only a finite number of nodes in $B(0,\delta)$ and an infinite number outside this ball. Since we are interested in the asymptotic behavior of $M_{v}(n)$, we can ignore all the traffic flow generated inside $B(0,\delta)$ for $n$ sufficiently large. Given $x$ and $z$ in $B(0,\delta)^{c}$, let $\gamma_{x,z}$ be the geodesic in $M$ connecting these two points. Since $M$ has negative curvature, the geodesics bend over to the origin. In particular, let $\omega$ be the number such that $\mathrm{vol}(B(0,\omega))=1$ and let $C_{y}^{E}$ be the shadow of this ball seen through the point $y$ to the unit circle according to the Euclidean metric and let $C_{y}$ be the shadow of this point according to the metric on $M$ as shown in Figure \ref{shadow}. Since $M$ has negative curvature, it is clear that $C_{y}^{E}\subset C_{y}$. Moreover, the set $C_{y}$ depends only on the distance from $y$ to the origin. Let $v(r)=\mathrm{length}(C_{y})$ in the unit circle $\partial M=S^{1}$.

\begin{figure}[!Ht]
  \begin{center}
    \centerline{\includegraphics{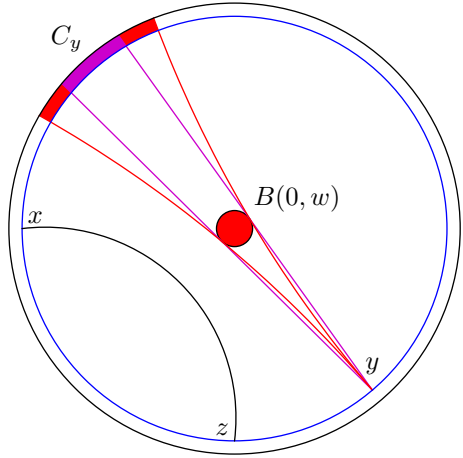}}
    \caption{Geodesic path and shadows in $M$. The shadow $C_{y}^{E}$ is in magenta and the shadow $C_{y}$ is in red.}
    \label{shadow}
  \end{center}
\end{figure}

By construction, the points in $\mathcal{P}$ are distributed in such a way that there is one point per unit area in $M$. Therefore, it is clear that the geodesics in $\mathcal{T}$ are quasi--geodesics in $M$. Then the traffic flow in $\mathcal{T}_{n}$ is equal to the traffic in $(B(0,r_n),ds)$ up to a constant. Hence, up to a fixed constant,
\begin{equation}
\limsup_{0<\delta\leq 1}\limsup_{n\to\infty}{\frac{M_{v}(n)}{N(n)^2}}\geq \limsup_{0<\delta\leq 1}\limsup_{n\to\infty}{\frac{T_{n}(0)}{N(n)^2}}\geq 2\pi\int_{\delta}^{1}{v(r)dr}>0
\end{equation}
proving our claim.
\end{proof}


As mentioned above, the rate of growth of $M_{v}(n)$ is between $\Theta(N)$ and $\Theta(N^2)$ for any graph with $N$ nodes. The next Theorem shows that this function can be narrowed down if the graph is planar. 

\begin{thm}\label{n3/2}
If $G$ is a planar graph, then 
\begin{equation}
\Theta(N^{3/2})\leq M_{v}\leq \Theta(N^2)
\end{equation}
where $N$ is the number of nodes in $G$. Moreover, the same result holds for any traffic routing (not necessarily geodesic routing).
\end{thm}

\begin{proof}
By the planar separator theorem (\cite{Alon}) in any $N$-vertex planar graph $G = (V,E)$, there exists a partition of the vertices of $G$ into three sets $A$, $S$ and $B$ such that each of $A$ and $B$ has at most $2N/3$ vertices, $S$ has $O(\sqrt{N})$ vertices, and there are no edges with one endpoint in $A$ and one endpoint in $B$. It is not required that $A$ or $B$ form connected subgraphs of $G$. The set $S$ is called the separator for this partition. Therefore, all the traffic between nodes in $A$ and $B$ has to pass through a node in $S$. Hence, by the pigeonhole principle there exists a node in $S$ with traffic at least $\Theta(N^{3/2})$.
\end{proof}

It was proved in \cite{load} that if the graph $G$ has exponential growth then 
\begin{equation}
\Theta(N^{2}/\log(N))\leq M_{v}\leq \Theta(N^2)
\end{equation}
independently of how the traffic is routed.

\subsection{Poisson Process Construction}
\label{SPconstruc}
In this subsection, we describe a way to construct a realization of the Poisson $\mathcal{P}$ in practice. Let 
\begin{equation}
\Phi: \R^2 \to \R^2
\end{equation}
be the measure preserving map given in polar coordinates by $\Phi(re^{i\theta})=\alpha(r)e^{i\theta}$ where $\alpha:[0,\infty)\to [0,t)$ is an increasing function that depends on the function $\rho$. This map is measure preserving if and only if for every $x>0$,
$$
\mathrm{vol_{\mathbb{R}^2}}(B(0,x))=\pi x^2 = 2\pi \int_{0}^{\alpha(x)}{\rho(y)y\,dy}=: 2\pi F(\alpha(x)).
$$
Therefore, $\alpha$ has to be equal to
\begin{equation}
\alpha(x)=F^{\langle-1\rangle}(x^2/2)
\end{equation} 
where $F^{\langle-1\rangle}$ is the inverse with respect to composition of the function $F$. Let $\mathcal{P}_{\mathbb{R}^2}$ be a Poisson point process on the Euclidean plane of constant intensity $1$. We can construct $\mathcal{P}$ as the set of points $\Phi(\mathcal{P}_{\mathbb{R}^2})$. Hence, everything boils down to constructing a Poisson process on the Euclidean space which can be done by partitioning the space into a set of disjoint annuli and uniformly distributing points in these annuli.

\section{Congestion Analysis and Degree Distribution}\label{analysis}

In this Section, we show some simulation results obtained for the maximum edge and vertex flow as well as the degree distribution for different densities $\rho$.

\subsection{Hyperbolic Plane with the Poincar\'e Disk Representation}
\label{Spoinc}
Assume that $\rho$ is as in Example~\ref{examp1}. Then it is an easy calculation to show that
$\int\rho(y)y\,dy=\frac{2\lambda}{1-y^2}$ and
$F(x)=\frac{2\lambda x^2}{1-x^2}$, hence
$$
\alpha(x)=\Bigg( \frac{x^2}{4\lambda+x^2} \Bigg)^{1/2}.
$$
In Figure \ref{FF1}, we show the Delaunay triangulation with 10,000 nodes of the previous surface for different values of the intensity $\lambda$. In table \ref{tab1} we show the diameter, average and maximum vertex and edge flow. We observe that as $\lambda$ increases, the maximum vertex and edge flows decrease but the averages increase. 

\begin{figure}[!Ht]
  \begin{center}
    \includegraphics[width=5cm]{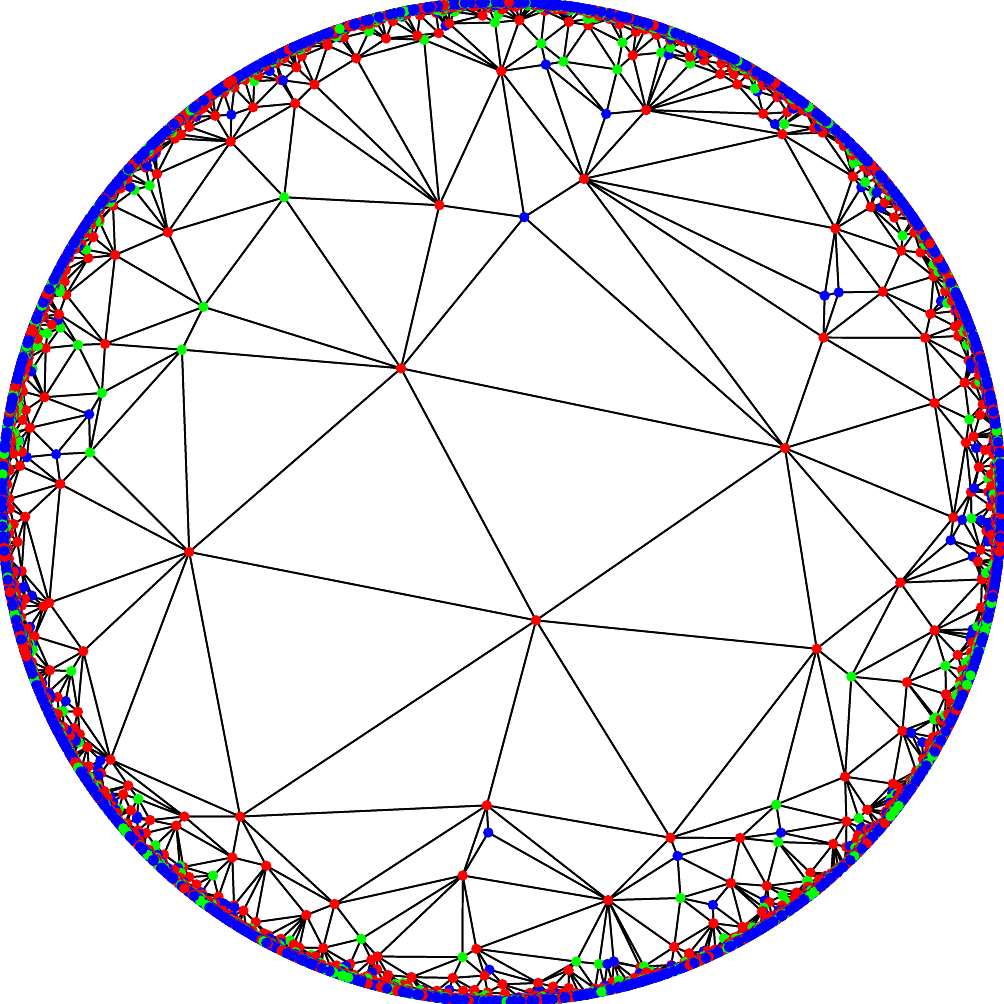}
    \includegraphics[width=5cm]{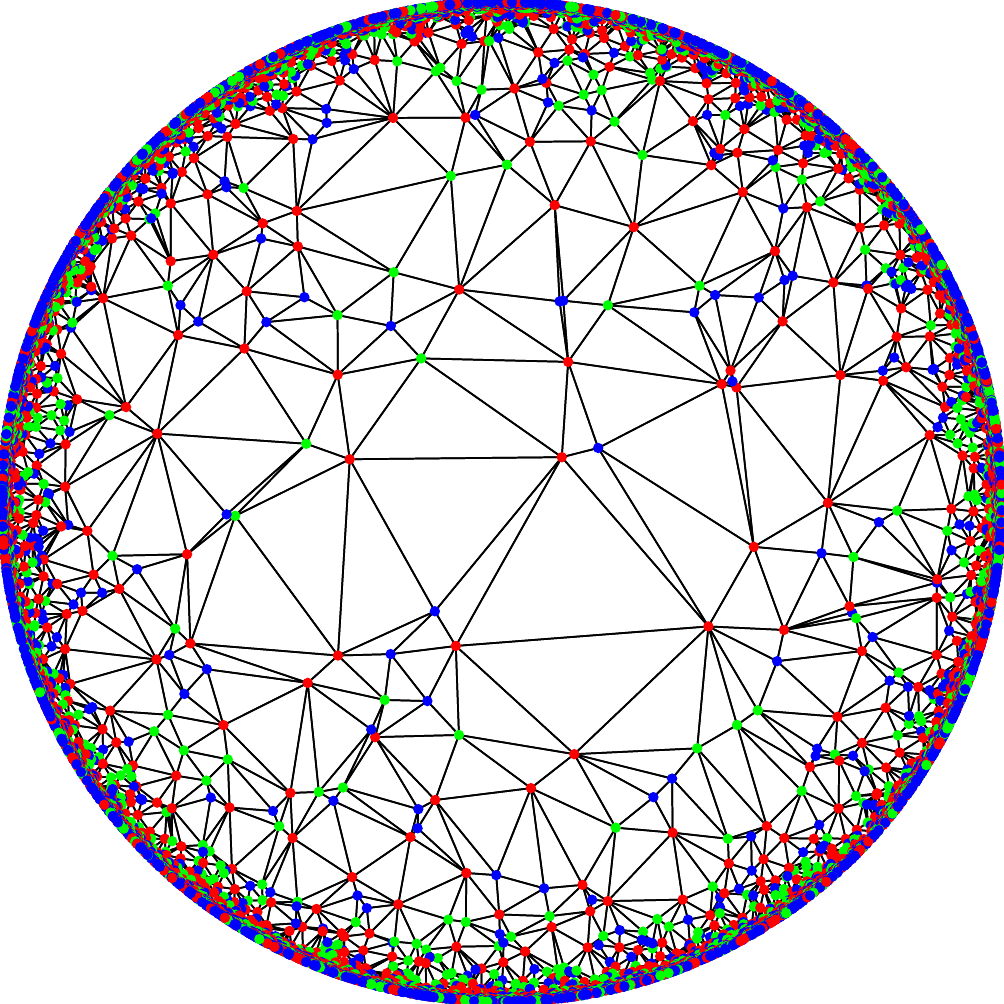}
    \includegraphics[width=5cm]{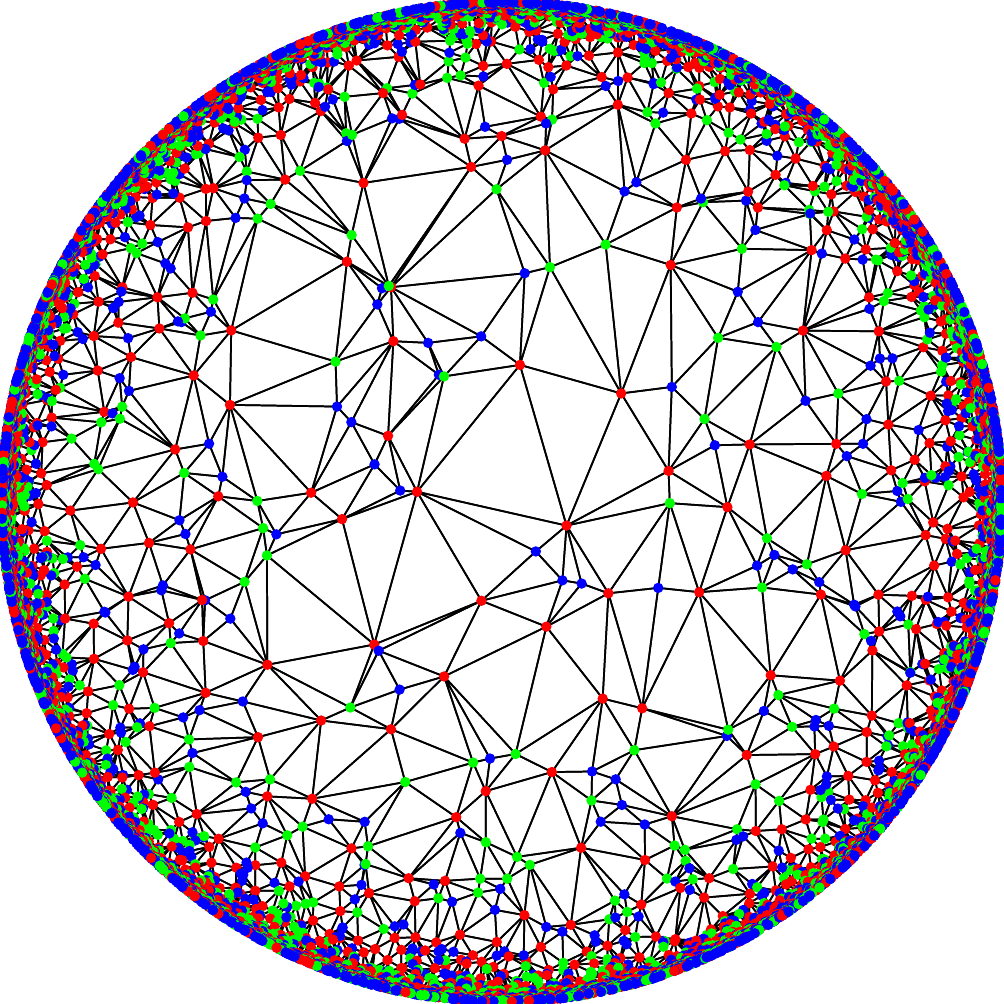}
    \caption{Delaunay triangulation with $\lambda$ equal to $1$ (left), $5$ (center) and $10$ (right).}
    \label{FF1}
  \end{center}
\end{figure}

\begin{table}[htp]
\caption{Traffic Congestion}
\label{tab1}
\begin{tabular}{|c|c|c|c|c|c|} \hline
$\lambda$ & Diameter&	AvgVflow&	MaxVflow&	AvgEflow&	MaxEflow \\ \hline
1& 20&	130764&	3.57E+07&	22031&	8.64E+06 \\ \hline
5& 31&	190195&	2.94E+07&	31906.8&	6.56E+06\\ \hline
10& 37&	221776&	2.50E+07&	37145.9&	7.53E+06\\ \hline																																
\end{tabular}
\end{table}

\subsection{Hyperbolic Space Representation in $\R^2$}
We can put a density $\rho$ in the entire $\R^2$ that corresponds to the hyperbolic metric on $\R^2$. More specifically, consider $\R^2$ with the metric 
$$
ds^2 = dr^2 + \frac{\sinh(\sqrt{-K}r)}{\sqrt{-K}}\,d\theta^2
$$
that corresponds to the hyperbolic space with uniform curvature $-K$. Hence, if $K=1$, we get the classical hyperbolic space. The area form is equal to $dA = \sinh(r)\,dr d\theta$. Then taking $\rho(r)=\frac{\sinh(r)}{r}$ we get a density that corresponds to the density to the Hyperbolic space density. In this case, 
$$
F(x) = \cosh(x)-1
$$
and it is easy to see that 
$$
\alpha(x) = \mathrm{acosh}(x^2/2+1).
$$

\subsection{A ``Generalized'' Hyperbolic Space}
\label{Sgenhy}
Let $t=1$ and $\rho(r) = \frac{2r^{(1-2a)/a}}{a(1-r^{1/a})^2}$ where $a>0$. Note that $a=1/2$ corresponds to the Poincar\'e disk. It can be show that
$\int\rho(y)y\,dy=\frac{2}{1-r^{1/a}}$ and
$$
F(x)=\frac{2x^{1/a}}{1-x^{1/a}}
$$
hence
$$
\alpha(x)=\Bigg(\frac{x^2}{4+x^2}\Bigg)^{a}.
$$
Figure \ref{FFz} gives a realization for $N=10,000$ and $a=0.01$ in comparison with the Euclidean density and the density corresponding to the Hyperbolic space representation in $\R^2$. Figure \ref{FFq} shows that the degree distribution of its corresponding Delaunay triangulation follows a power law for $a=0.01$ and $0.05$. Moreover, the degree of the power law distribution can be tuned up by choosing the appropriate intensity. Therefore, this is a way to construct a random triangulation with a power law degree distribution. Not surprisingly, $a=0.5$, which correspond to the Poincar\'e disk, does not follows a power law distribution.
\begin{figure}[!Ht]
  \begin{center}
    \includegraphics[width=5cm]{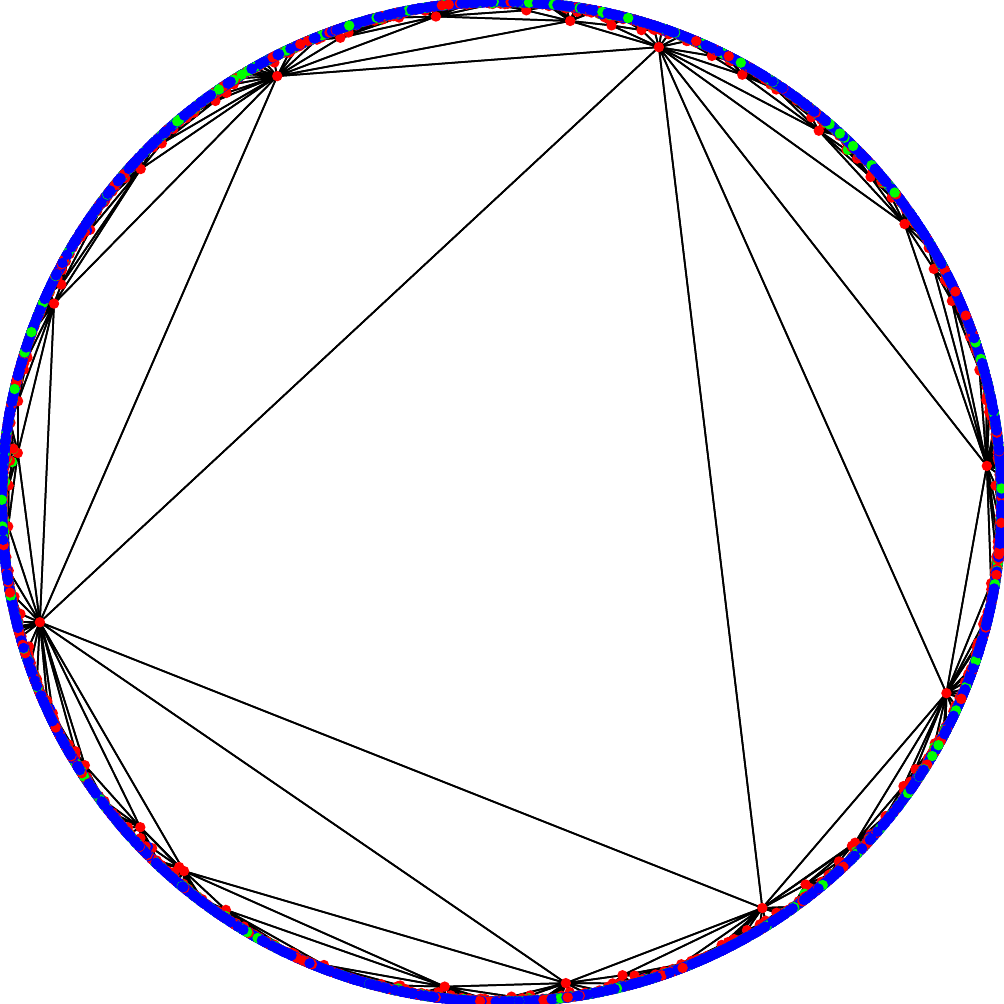}
    \includegraphics[width=5cm]{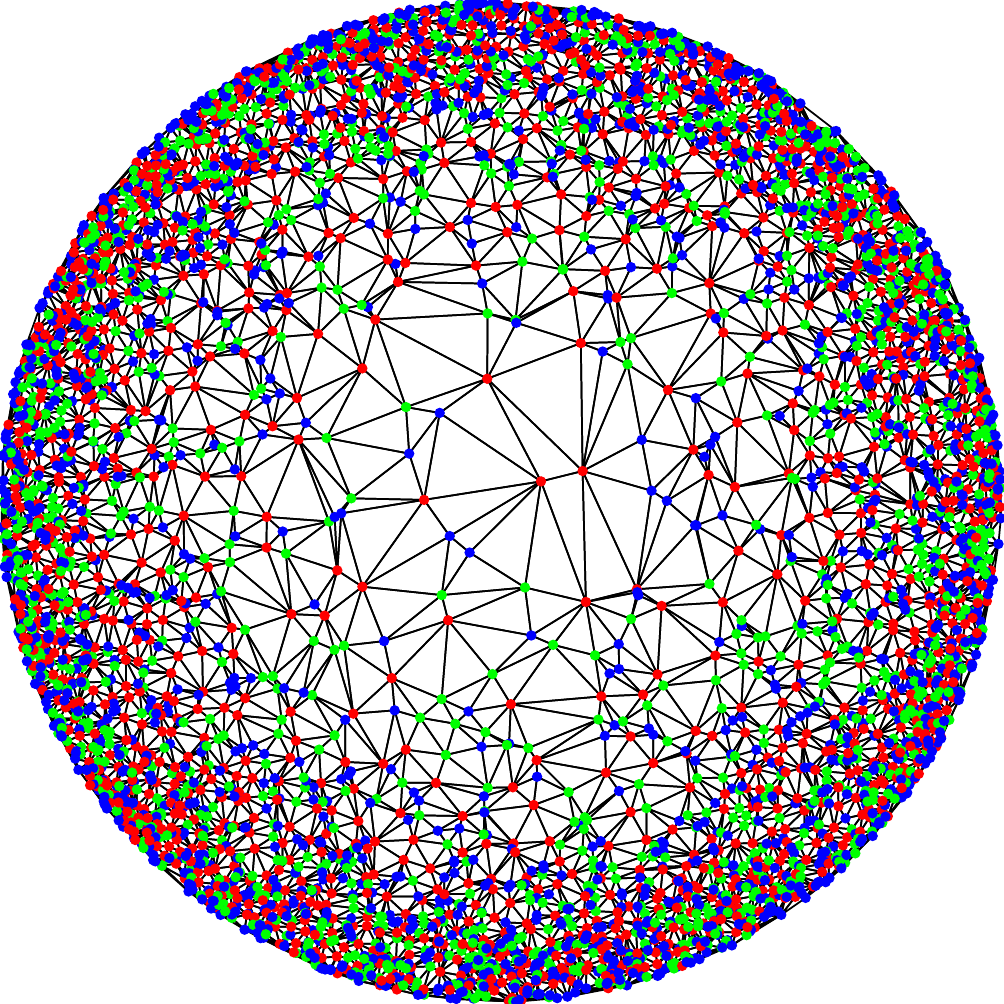}
    \includegraphics[width=5cm]{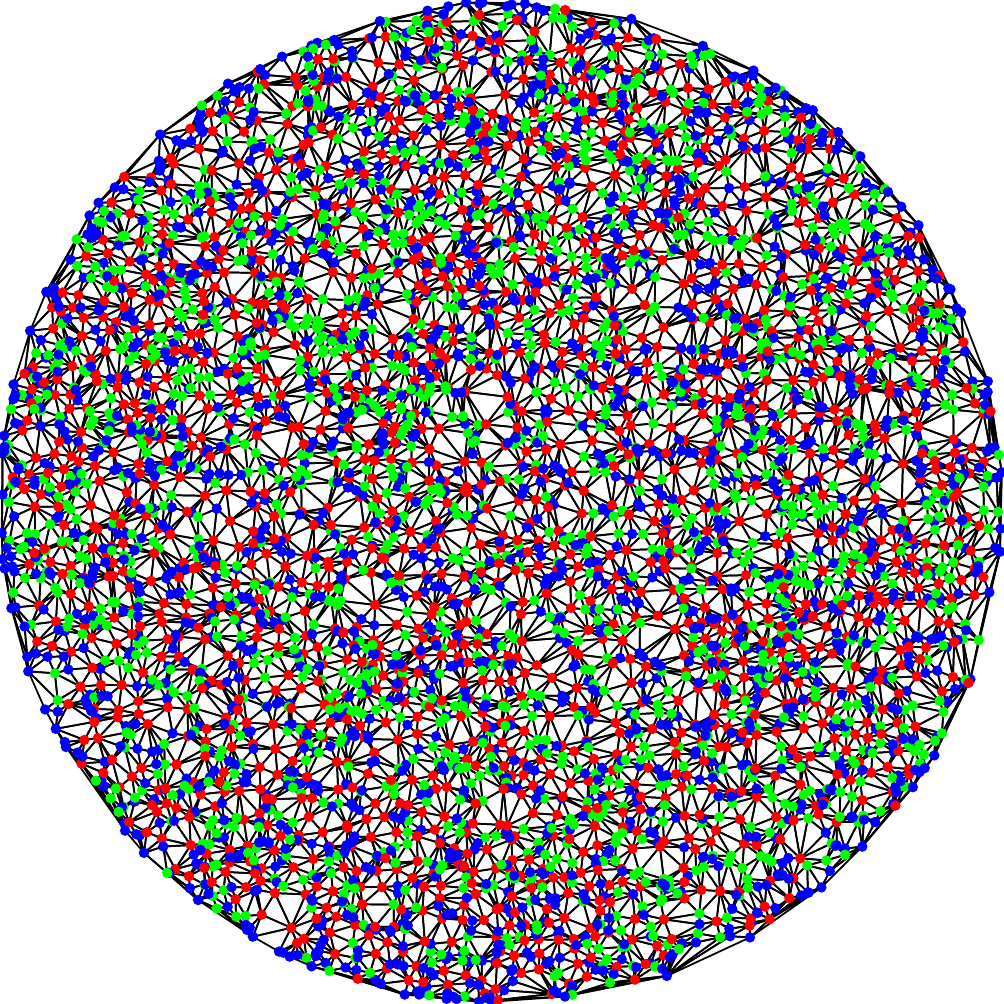}
    \caption{Delaunay triangulation for $N=10,000$ and $a=0.01$ (left). Case Hyperbolic space representation in $\R^2$ (center) and Euclidean case (right).}
    \label{FFz}
  \end{center}
\end{figure}

\begin{figure}[!Ht]
  \begin{center}
    \includegraphics[width=12cm]{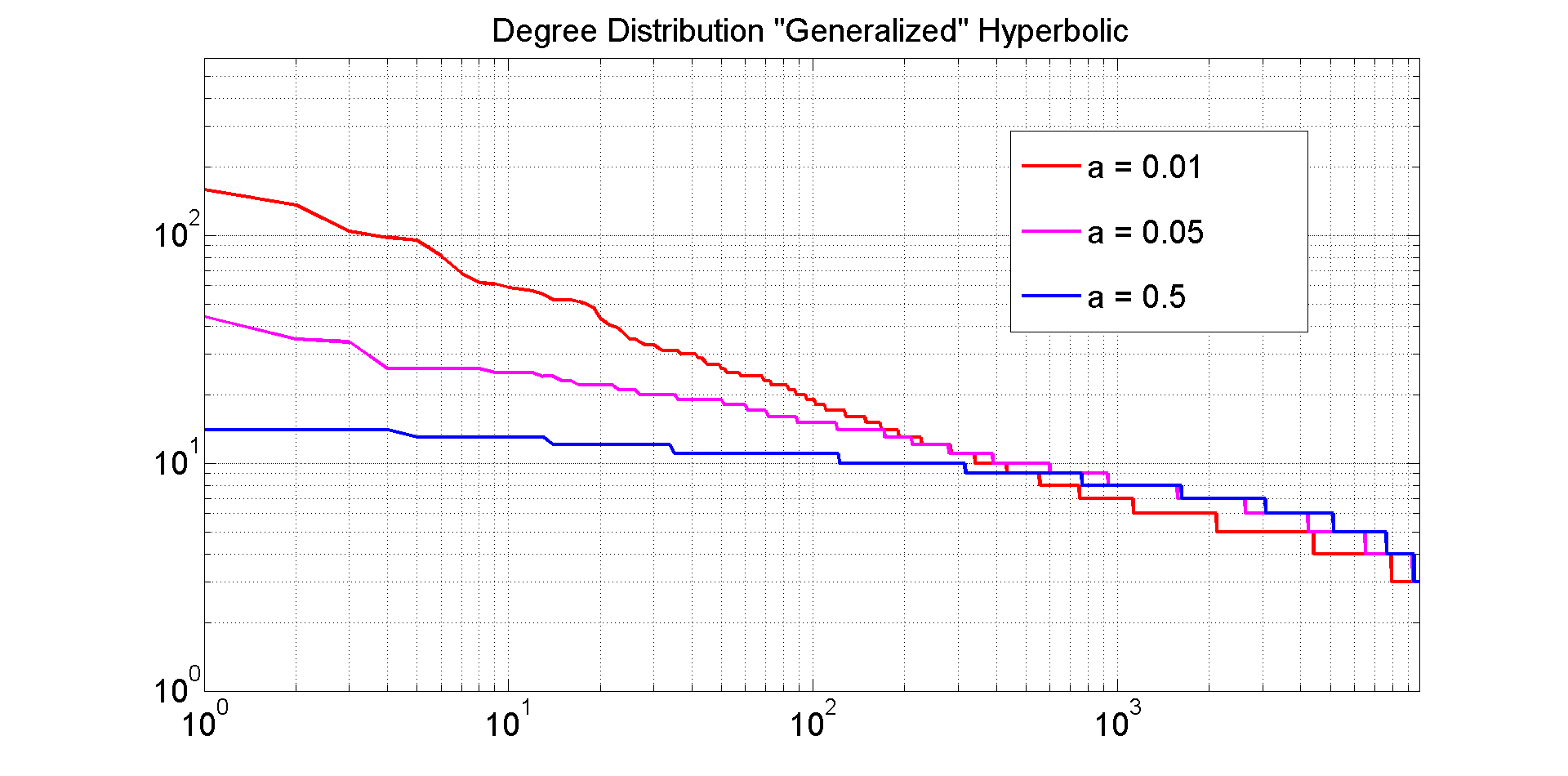}
    \caption{Degree distribution for $N=10,000$ and $a=0.01$, $a=0.05$ and $a=0.5$.}
    \label{FFq}
  \end{center}
\end{figure}

\subsection{Another Example}
Let $t=1$ and $\rho(r) = \frac{\lambda^2}{(1-r)^3}$ for $\lambda>0$. Then it can be show that
$\int\rho(r)r\,dy=\frac{\lambda^2(r-1/2)}{(1-r)^2}$ and
$F(x)=\frac{\lambda^2}{2}\frac{x^2}{(1-x)^2}$, hence
$$
\alpha(x)=\frac{x}{\lambda+x}.
$$
Figure \ref{FF2} shows the Delaunay triangulation of this for different values of the intensity $\lambda$.

\begin{figure}[!Ht]
  \begin{center}
  	\includegraphics[width=5cm]{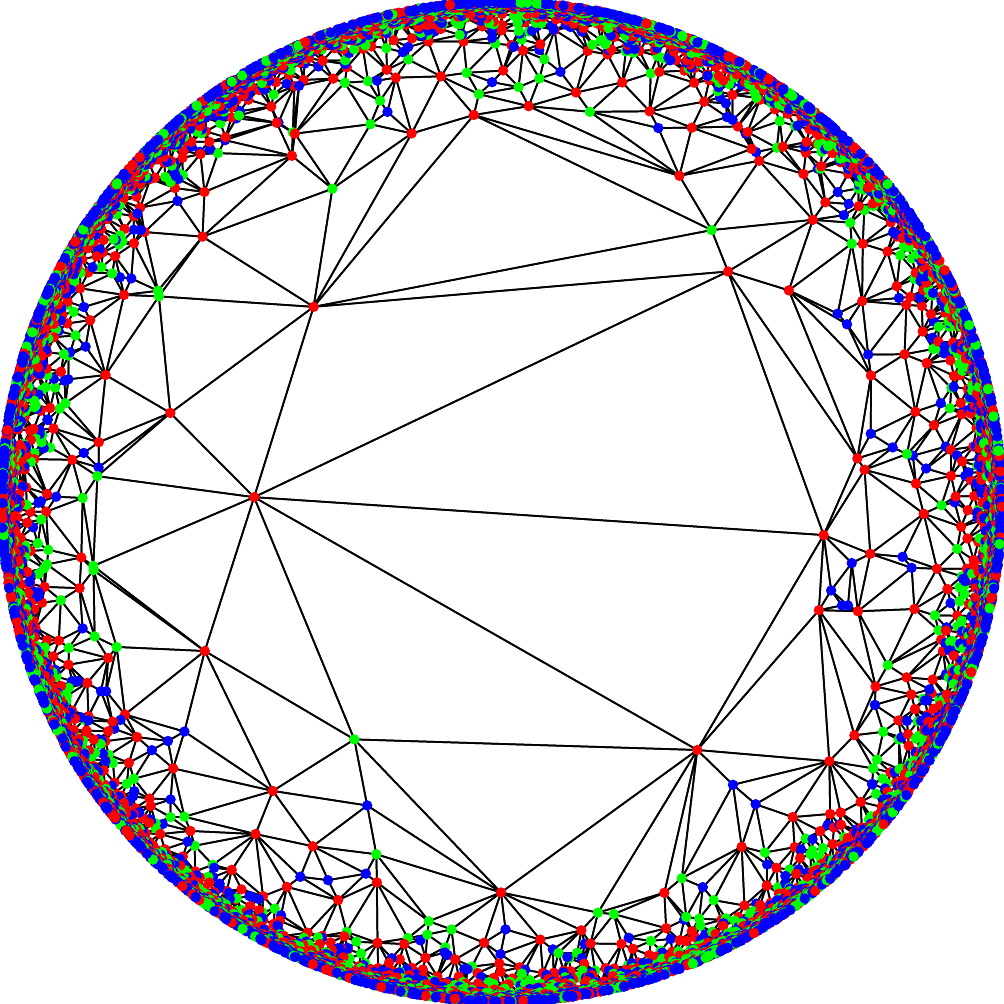}
    \includegraphics[width=5cm]{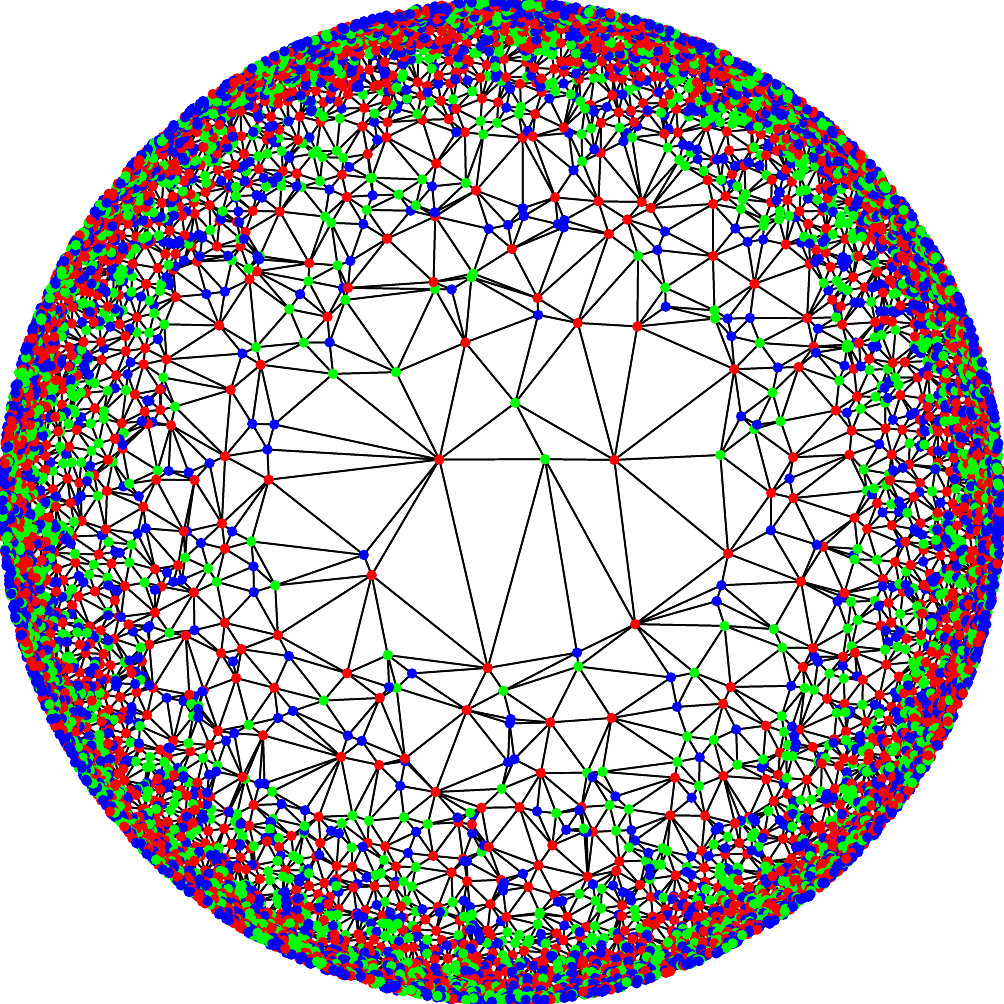}
    \includegraphics[width=5cm]{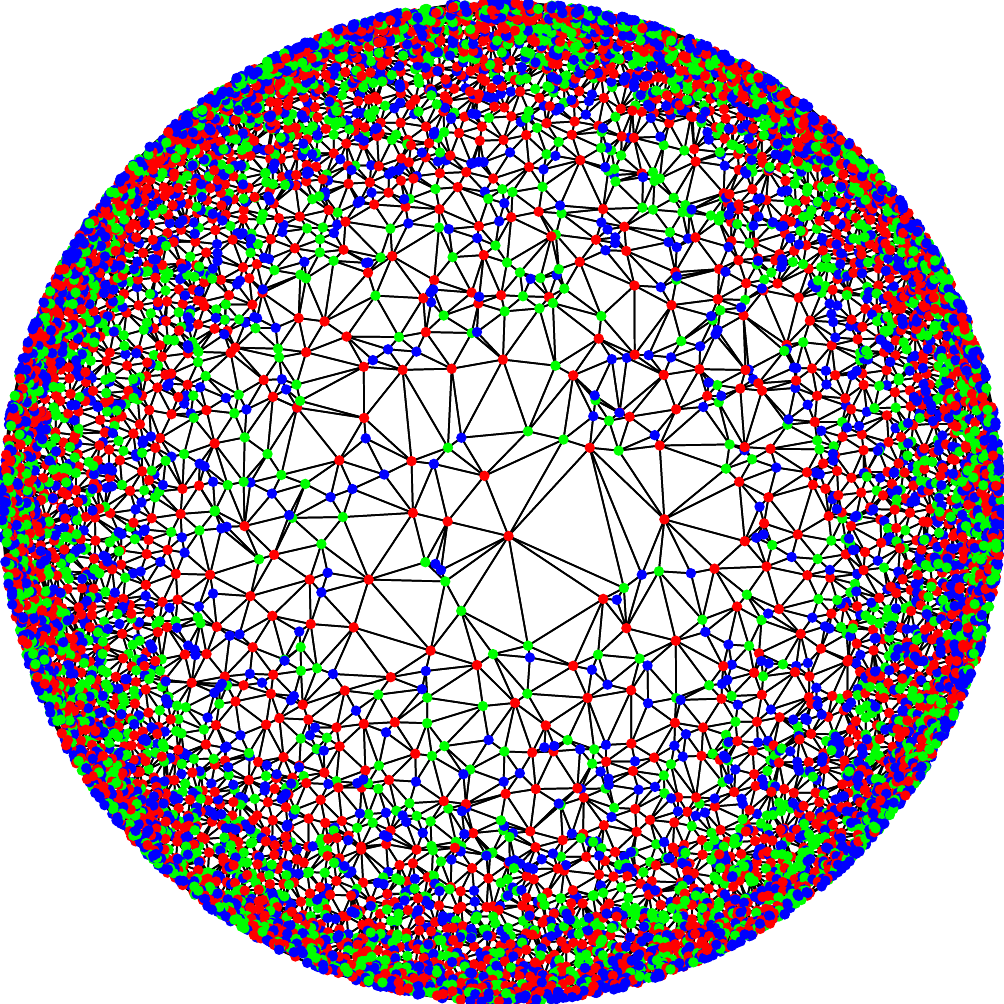}
    \caption{Delaunay triangulation with intensity $0.1$ (left), intensity $1$ (center) and intensity $10$ (right).}
    \label{FF2}
  \end{center}
\end{figure}

\subsection{Slow Growing}
Let $t=\infty$ and $\rho(r)=\frac{1}{r(r+1)}$. Then 
$$
F(x) = \int_{0}^{x}{\rho(y)y\,dy} = \ln(x+1). 
$$
Hence, 
$$
\alpha(x)=\exp\Bigg(\frac{x^2}{2}-1\Bigg).
$$

In Figure \ref{vd}, we see the asymptotic growth of the maximum vertex flow and the maximum edge flow as a function of the number of nodes $N$ of $\mathcal{T}_{n}$ for the densities
from Subsections \ref{Spoinc}--\ref{Sgenhy}. It can be shown that by changing the function $\rho$, the maximum vertex congestion changes from any possible growth rate between $\Theta(N^{3/2})$ and $\Theta(N^2)$. As discussed in Theorem \ref{n3/2}, these are the only allowed growths for planar graphs.

\begin{figure}[!Ht]
  \begin{center}
  	\includegraphics[width=8cm]{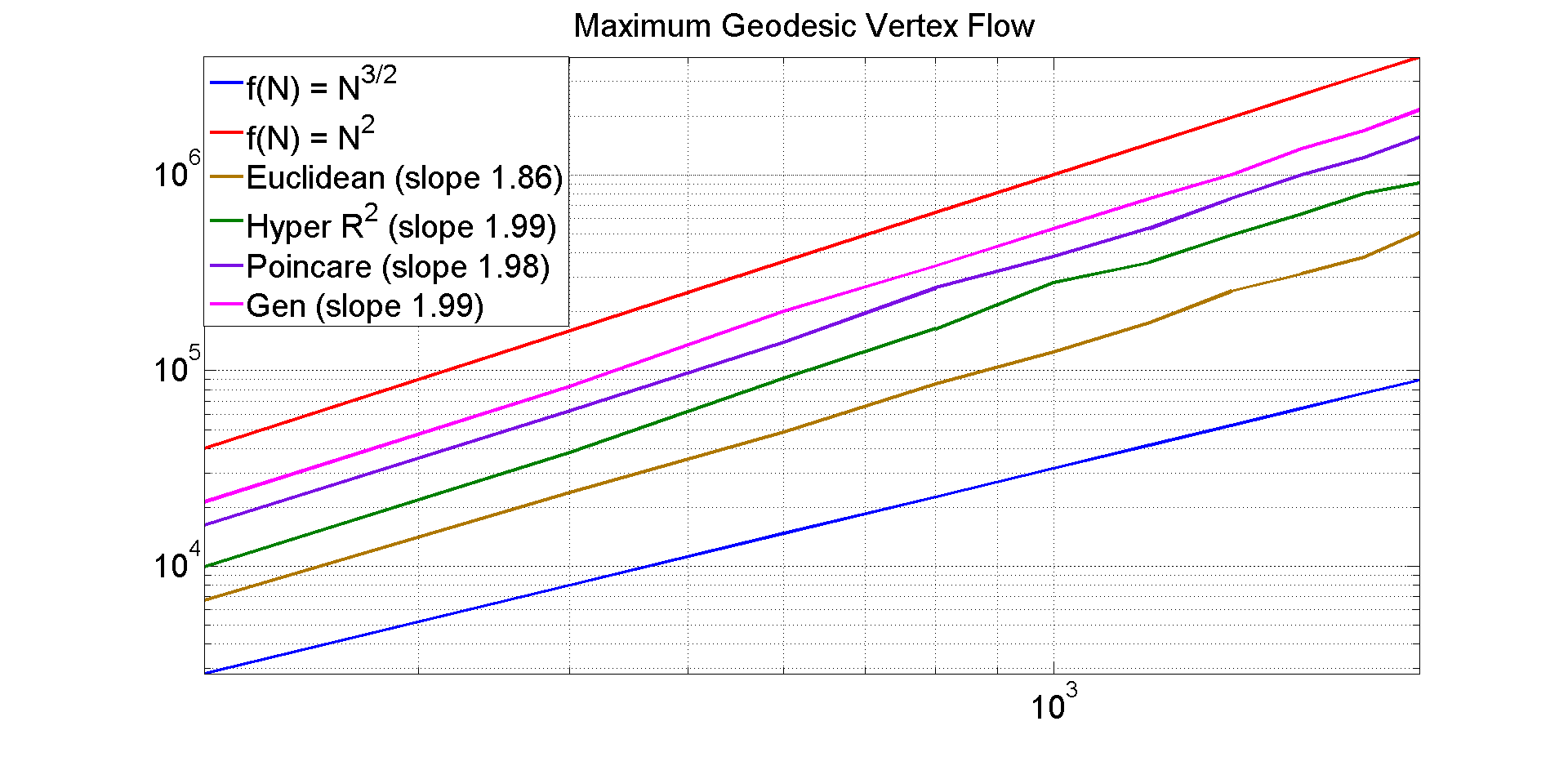}
  	\includegraphics[width=8cm]{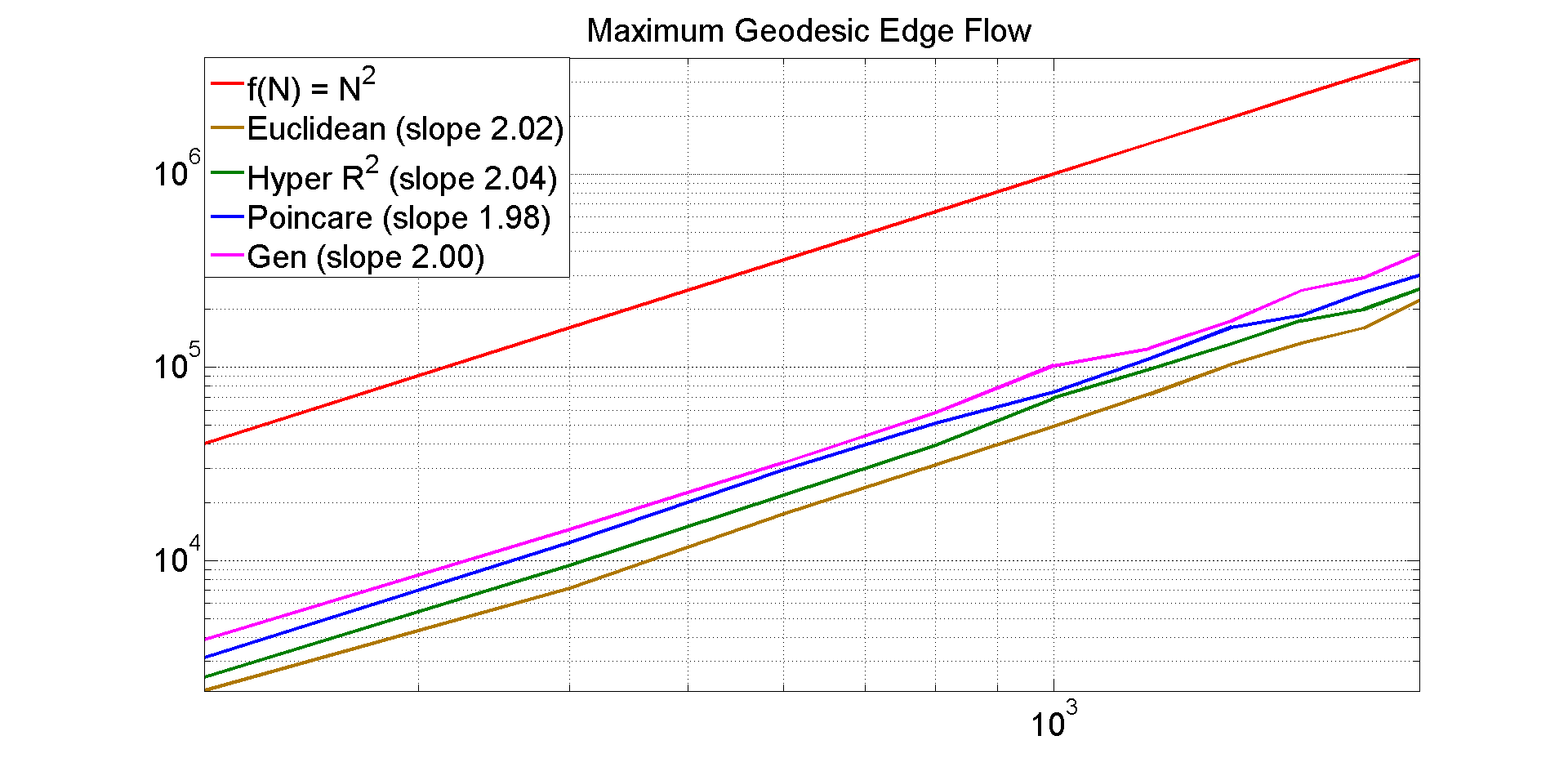}
    \caption{Log-log plots of the maximum vertex congestion for different densities as a function of the number of nodes, averaged over 20 trials.}
    \label{vd}
  \end{center}
\end{figure}

\section{Erd\"os-Renyi Graphs, Random Geometric Graphs and Expanders}\label{rg}

In this Section, we explore the maximum vertex congestion with geodesic routing for Erd\"os-Renyi, random geometric graphs and expanders. More specifically, let $\mathrm{ER}(N,p_N)$ be the random Erd\"os-Renyi graph with $N$ nodes and $p_{N}=2\log(N)/N$. Note that these graphs are connected almost surely since the probability $p_N$ is slightly above the connectivity threshold $(1+\epsilon)\log(N)/N$. Let $\mathrm{RGG}(N,r_N)$ be the random geometric graph in the unit square with $n$ nodes and radius $r_N=\sqrt{2\log(N)/N}$. These graphs again are connected since the radius is slightly above the connectivity radius. Now we turn to a particular construction of random expander graphs. We will assume that $k=2l$ is an even integer. To build a $k$-regular graph on $N$ vertices $\{1,\ldots,N\}$, what we do is pick $l$ permutations $\pi_1,\ldots,\pi_l: \{1,\ldots,N\} \rightarrow \{1,\ldots,N\}$, and let $\mathcal{G}_{k,N}$ be the graph formed by connecting $v$ to $\pi_i(v)$ for all $v \in \{1,\ldots,N\}$ and $i=1,\ldots,l$. Formally, this is not always a $k$-regular graph since there could be multiple edges and loops. It is well known that eliminating these problematic edges the graph is a $k$-regular
and an expander with probability tending to $1$ as $N\to\infty$.

Figure \ref{er_rgg}, shows the maximum vertex flow and the maximum edge flow as a function of the number of nodes for these three graph families. In the case of the expander family, we use $k=6$. As can be appreciated from these plots, there is very little congestion even for geodesic routing. Moreover, note that at the connectivity threshold is the worst case scenario since if we enlarge $p_N$ or $r_N$ the congestion decreases even more. We should also point out that the expander graphs have the least edge and vertex congestion even though these graphs have the least number of edges among the three families.

\begin{figure}[!Ht]
  \begin{center}
  	\includegraphics[width=8cm]{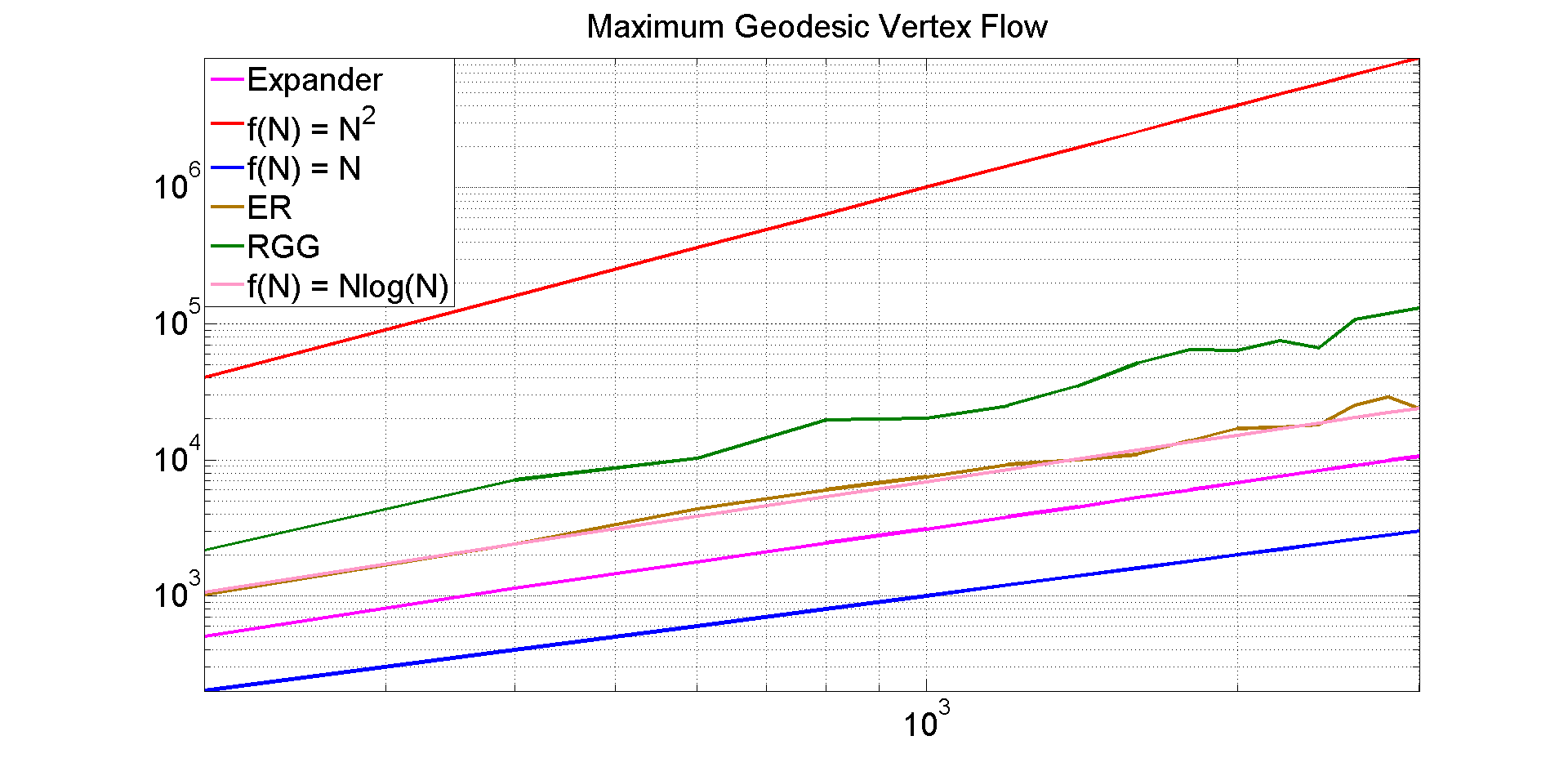}
  	\includegraphics[width=8cm]{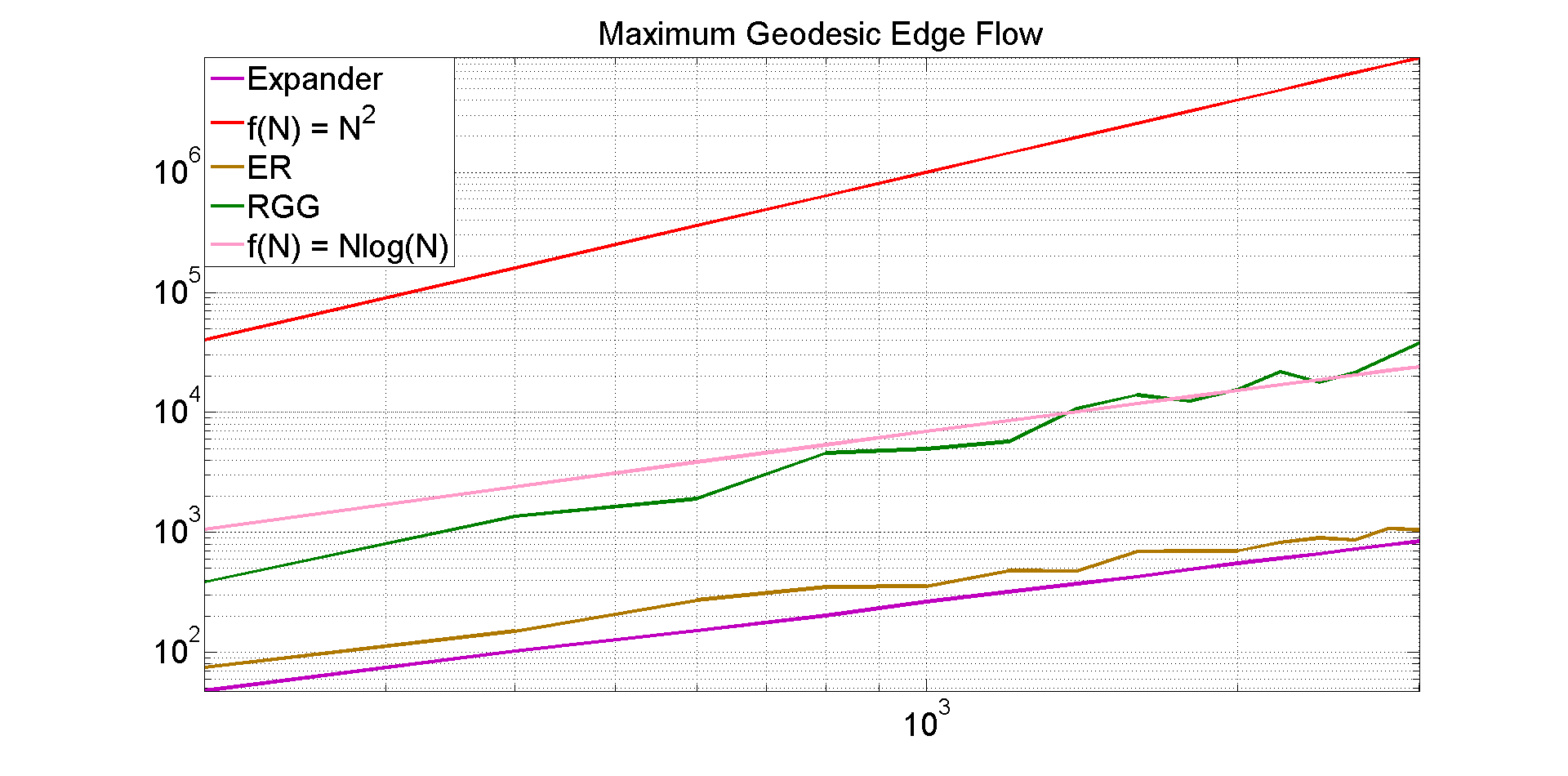}
    \caption{Log-log plot of the maximum vertex and edge congestion as a function of the number of nodes and averaged over 20 realizations.}
    \label{er_rgg}
  \end{center}
\end{figure}

\par As discussed in Section~\ref{Sconstruc}, Gromov hyperbolic graphs have congestion of the order $\Theta(N^2)$. In particular, any $k$--regular tree (also called a Bethe lattice) has highly congested nodes. More precisely, the following result was proved in \cite{Tucci1}.
\begin{prop}
For a $k$--regular tree with $N$ nodes, 
\begin{equation}
M_v = \frac{k-1}{2k}\cdot(N-1)^2+N-1.
\end{equation}
Furthermore, the maximum congestion occurs at the root.
\end{prop}

Note that trees are some of the most congested graphs one can consider. The reason is that much of the traffic must pass through the root of the tree. This leads to the following question.
\begin{question}
Is there a way to reduce the vertex and edge congestion in a graph by adding a relatively small number of edges? 
\end{question}
We answer this question positively, but we first need the following Lemma.
\begin{lem}\label{LL}
Let $G$ be a graph with bounded geometry, i.e. $\sup_{v\in G} \mathrm{deg}(v)\leq \Delta<\infty$. Then for every $v\in G$, the traffic flow passing through $v$ satisfies
\begin{equation}
T(v)\leq \Delta^2(\Delta-1)^{D-2}\cdot D^2
\end{equation}
where $D=\mathrm{diam}(G)$.
\end{lem}

\begin{proof}
Let $v\in G$ and define $S_{k}:=\{x\in G: d(v,x)=k\}$. Then it is clear that $G =\{v\}\cup\bigcup_{p=1}^{D}S_{p}$ and moreover $|S_{k}|\leq \Delta(\Delta-1)^{k-1}$. Furthermore,
$$
T(v)\leq \sum_{k+l\leq D}{|S_{k}|\cdot|S_{l}|}
$$
where the inequality is coming from the fact that if $k+l>D$ then the geodesic path between a node in $S_{k}$ and a node in $S_{l}$ does not pass through $v$. Hence,
$$
T (v)\leq \sum_{k+l\leq D}{\Delta^2(\Delta-1)^{k+l-2}}\leq \Delta^2(\Delta-1)^{D-2}D^2.
$$
\end{proof}

Given a graph $G$, a matching $M$ in $G$ is a set of pairwise non-adjacent edges; that is, no two edges share a common vertex. The following result is due to Bollobas and Chung \cite{Bol}. 

\begin{thm}\label{Thm4x4}
Suppose $G$ is a graph on $N$ vertices with bounded degree $k$ satisfying the property that for any $x\in V(G)$, the $i$-th neighborhood $N_i(x)$ of $x$ (i.e., $N_{i}(x)=\{y:d(x,y)=i\}$) contains at least $ck(k-1)^{i-2}$ vertices for $i\leq(1/2 + \epsilon)\log_{k-1}(N)$, where
$\epsilon$ and $c$ are fixed positive values. Then by adding a random matching to $G$ gives
a graph $H$ whose diameter $D(H)$ satisfies
$$
\log_{k}(N)-C\leq D(H)\leq \log_{k}(N) + \log_{k}\log(N)+ C
$$
with probability tending to $1$ as $N$ goes to infinity, where $C$ is a constant depending on $\epsilon$
and $c$.
\end{thm}

Now we are ready to answer the previous question. 

\begin{thm}
Let $G$ be a graph on $N$ vertices with bounded degree $k$ satisfying the hypothesis of Theorem~\ref{Thm4x4}. Let $H$ be the graph constructed by adding a random matching to $G$. Then the maximum vertex flow $M_{v}(N)$ with geodesic routing on $H$ satisfies
\begin{equation}
M_{v}(N)\leq k^{C}N\log(N)\log_{k}^2(N) + o(N\log(N)\log_{k}^2(N)).
\end{equation}
with probability tending to $1$ as $N\to\infty$.
\end{thm}

\begin{proof}
Applying Lemma \ref{LL} gives
$$
M_v(H)\leq k^{D(H)}D(H)^2.
$$
Since
$$
D(H)\leq \log_{k}(N) + \log_{k}\log(N)+ C
$$
with high probability, we see that
\begin{eqnarray*}
M_v(H) &\leq & (\log_{k}(N) + \log_{k}\log(N)+ C)^{2}k^{\log_{k}(N) + \log_{k}\log(N)+ C}\\
& = & k^{C}N\log(N)\log_{k}^2(N) + o(N\log(N)\log_{k}^2(N)).
\end{eqnarray*}
\end{proof}

In particular, in every $k$--regular tree, we can reduce the congestion from $\Theta(N^2)$ to $O(N\mathrm{polylog}(N))$ by adding at most an extra edge for every node. The same applies for the Delaunay triangulations constructed in the previous sections.

It was proved in \cite{BVega} that the diameter of a random $k$--regular satisfies the following property.

\begin{thm}\label{Boll}
For $k\geq 3$ and sufficiently large $N$, a random $k$--regular graph $G$ with $N$ nodes has diameter at most  
$$
\log_{k-1}(N) + \log_{k-1}(\log_{k-1}(N))+C
$$
where $C$ is a fixed constant depending on $k$ and independent on $N$.
\end{thm}

Hence by reasoning as in the proof of Theorem~\ref{Boll} we obtain the following corollary.

\begin{thm}
Let $G$ be a random $k$--regular graph. Then the maximum vertex flow $M_{v}(N)$ with geodesic routing on $G$ is smaller than
\begin{equation}
k^{C}N \log_{k}^3(N) + o(N\log_{k}^3(N)).
\end{equation}
with probability tending to $1$ as $N\to\infty$.
\end{thm}

In Figure \ref{bl}, we see the maximum vertex and edge flow for a random regular tree of degree 6 (Bethe Lattice of degree 6) as a function of the number of nodes, and the same after adding one or two random matchings to the edge list of this graph. By adding one random matching to the edge list, we reduce the flow from order $N^2$ to $N\log(N)^2$. 

\begin{figure}[!Ht]
  \begin{center}
  	\includegraphics[width=8cm]{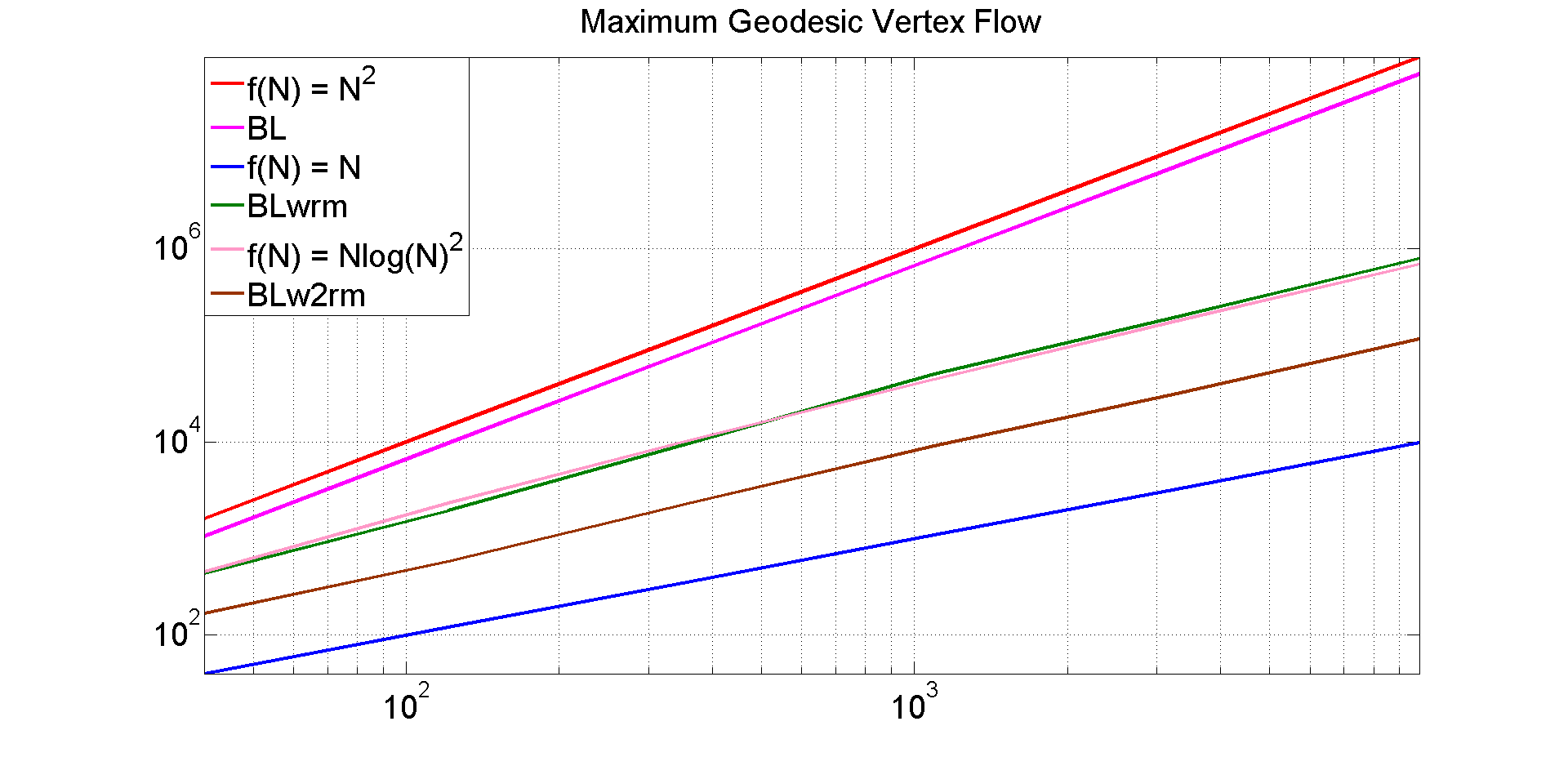}
  	\includegraphics[width=8cm]{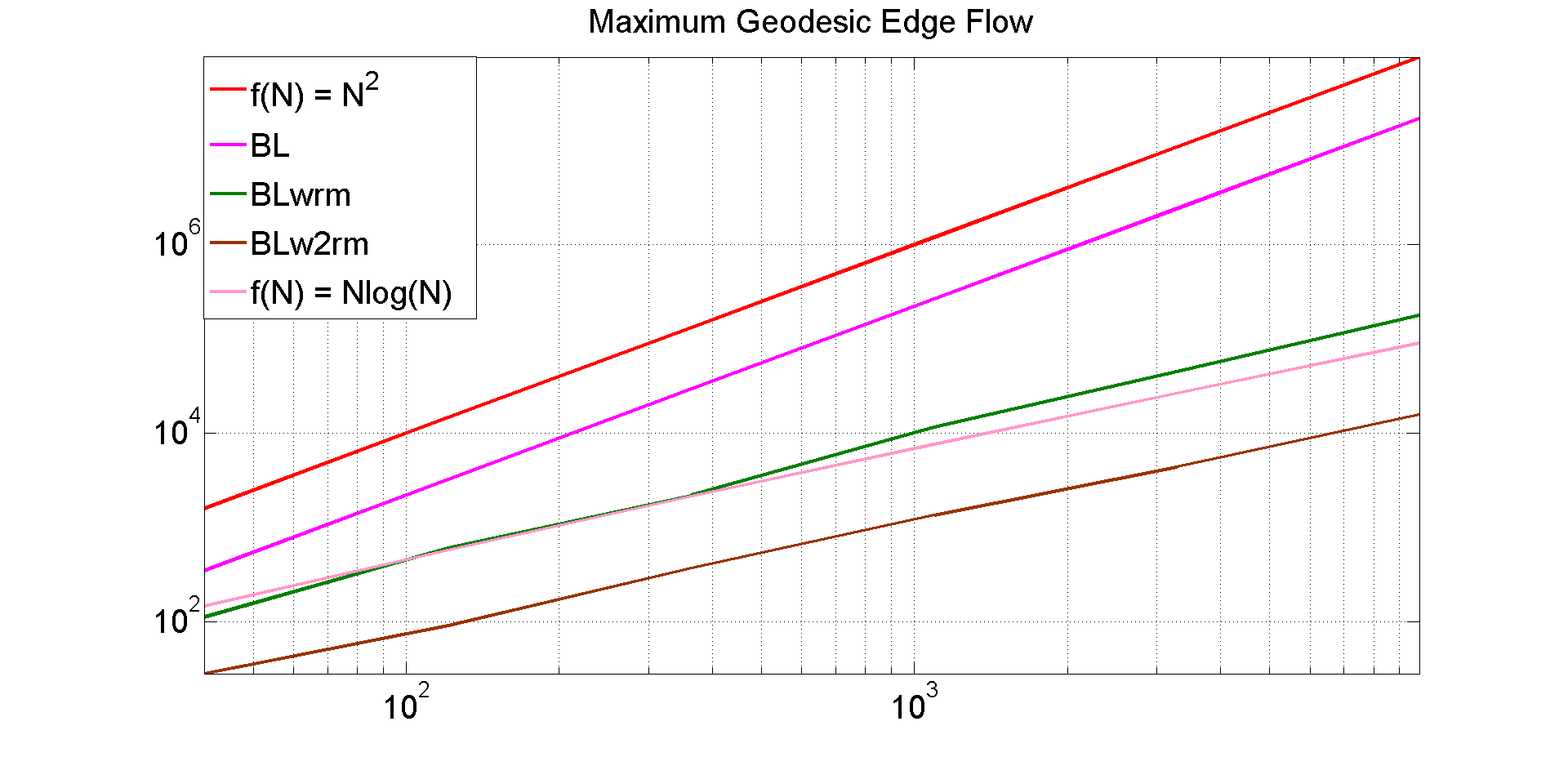}
    \caption{Edge and vertex flow for a Bethe lattice of degree 6 as a function of the number of nodes, and the same after adding one or two random matchings.}
    \label{bl}
  \end{center}
\end{figure}

\section{Implementation and Algorithms}\label{imp_alg}

Section \ref{SPconstruc} explains that we use a Poisson process to generate random
points and remap the distance from the origin via $r\rightarrow\alpha(r)$.
Then we need the Delaunay triangulation with a flow computation based on routing
over geodesic paths.  Thus the main computation can be broken down into four steps:
\begin{enumerate}
\item Generating appropriately distributed pseudorandom vertices,
	\label{STvtxgen}
\item Computing the Delaunay triangulation,
	\label{STdelaun}
\item Finding the shortest path between each pair of points,
	\label{STshpaths}
\item Computing the traffic flow.
	\label{STflow}
\end{enumerate}
The purpose of this section is to describe these steps, consider their run time
and space complexity, and study the efficient implementation of the critical
the ones that consume the most time and space.

Steps \ref{STvtxgen} and \ref{STdelaun} can be used alone to get degree statistics
for graphs that are too big for the traffic flow computation, so it is noteworthy
that Step~\ref{STvtxgen} handles an $n$-vertex graph in time $O(n)$, while
Step~\ref{STdelaun} requires $O(n\log n)$.

Step \ref{STvtxgen} simply selects pseudorandom points $(x_i,y_i)$ uniformly
in a disk, converts to polar coordinates, replaces $r$ by $\alpha(r)$, and converts
back to rectangular coordinates.  If the Poisson rate suggests $n$ points,
we pseudorandomly select a number $N$ from the rate~$n$ Poisson distribution,
and select $N$ points in the disk.
For maximum flexibility, our C$++$ implementation allows the user to enter
the function $\alpha(r)$ in postfix notation using arithmetic operations and
standard transcendental functions.  This adds a constant factor overhead
to the run time for Step~\ref{STvtxgen}, but this is subsumed by Step~\ref{STdelaun}.

The Delaunay triangulation for Step~\ref{STdelaun} is computed via Fortune's
sweep line algorithm running as a separate process \cite{fo:svor}.

Step~\ref{STshpaths} could potentially use a lot of time and space, so efficiency
is important.  The all-points shortest path algorithm just takes the graph's
adjacency matrix~$A$ and computes $A^p$ for a power $p$ not less than the
graph diameter using min-plus arithmetic  For our application, $A$ is very
sparse and $p$ is only $O(\log n)$, so there is no harm in exponentiating
via repeated multiplication by $A$.  (The process stops when such a multiplication
has no effect.)

The matrix for the all-points shortest path problem keeps track of shortest known
path lengths, and sparseness means having a lot of infinite entries.
As suggested by Figure~\ref{Flfacs}, it is important for the min-plus matrix
multiplication to take advantage of sparsity, yet be efficient when the matrix is
dense.  Hence, we store it as a dense (integer) matrix, but also have a bitmap
that keeps track of sparsity.  A big switch statement can consider 8 entries
at a time, and do 0--8 add and minimize operations as needed.  Furthermore,
one 32-bit boolean operation can cause 32 entries to be skipped if the matrix
is very sparse.  It also helps to compute the load factor and switch to dense-matrix
arithmetic when it gets high.

\begin{figure}[htp]
\centerline{\includegraphics{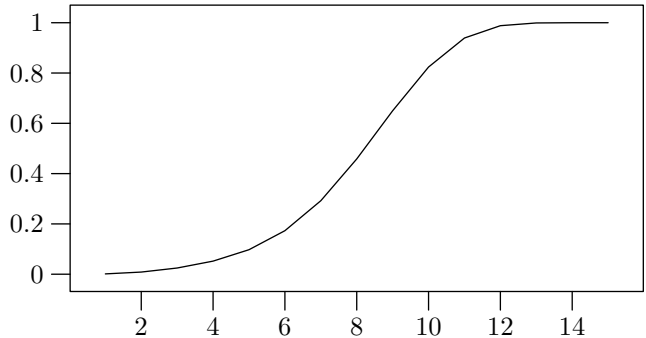}}
\caption{The matrix load factor at each step in the all-points shortest path computation
	for an $n=800$ problem with diameter 14.}
\label{Flfacs}
\end{figure}

Let $D$ be the matrix that gives the path lengths from the all-points shortest path
computation.  The shortest-path routing that forms the basis for Step~\ref{STflow} requires
flow from a vertex~$v$ destined for
a vertex~$w$ to go to neighbors $v_i$ of $v$ for which $D_{v_iw}=D_{vw}-1$.
When there is more than one such neighbor, we divide the flow into equal
parts.  Hence, each edge from $v$ has a sparse vector that tells, for each possible
destination, what fraction of $v$'s outbound traffic takes that edge.
These \emph{redistribution vectors} are sparse, and their non--zero entries are of
the form $1/k$ for small positive integers~$k$.  Table~\ref{Trdistk} tells
how these $k$ values are typically distributed.

\begin{table}[htp]
\caption{Popularity of $k$ values in redistribution vectors}
\label{Trdistk}
\begin{tabular}{|c|c|c|c|c|c|c|c|} \hline
1&     2&     3&     4&      5&      6&       7&      $\ge8$\\ \hline
49\%&  33\%&  12\%&	 3.9\%&  1.0\%&  0.57\%&  0.13\%&  0\%\\ \hline
\end{tabular}
\end{table}

Table~\ref{Trdistk} also provides an estimate of the sparsity in a typical redistribution
vector.  Since $1\cdot0.49+2\cdot0.33+\ldots 7\cdot0.0013=1.76$ and the tabulated
data are for a planar graph with average degree~6, the load factor is $1.76/6=0.293$.
Hence one byte almost always suffices to give a non--zero entry and specify how many
zeros since the last non--zero.

The traffic flow computation for Step~\ref{STflow} begins with a starting flow
matrix $S$ such that each $S_{i,j}$ gives the initial specification for flow
from vertex~$i$ to vertex~$j$.  The object is to compute an overall flow matrix
\begin{equation}
 F = S + \omega S + \omega^2 S + \omega^3 S + \cdots,
\label{EflowF}
\end{equation}
where $\omega$ is a linear operator formed by componentwise multiplication of
columns of a flow matrix times redistribution vectors, e.g., the redistribution
vector $r_{ij}$ for an edge $i,j$ adds $\mathop{\mathrm{diag}}(r_{ij})$ times column~$i$ of
$\omega$'s argument matrix to column~$j$ of $\omega$'s result matrix.
Note that (\ref{EflowF}) specifies a fixed point of the iteration $F\gets S+\omega F$.
This fixed point is reached after $O(\log n)$ steps because
$\omega^p$ is the zero operator whenever exponent $p$ exceeds the graph diameter.

To compute some column~$j$ of $S+\omega F$, we need redistribution vectors $r_{ij}$
for every $i\in E_j$, where $E_j$ is the set of all $i$ such that $i,j$ is in the
edge set~$E$:
$$ S_j + \sum_{i\in E_j} \mathop{\mathrm{diag}}(r_i,j) F_i. $$
This is trivial, except that we can reduce the need for scratch memory by
copying the result back to $F_j$ as soon as the old $F_j$ is no longer needed.

Now consider the total memory requirement for the flow computation on a graph of
$n$~vertices and $m$~edges.  The all-points shortest path matrix~$D$ (including sparsity bits)
requires $(4+\frac{1}{8})n^2$ bytes.  Since $i,j$ and $j,i$ require different
redistribution vectors at about $0.3n$ bytes each, the total memory for redistribution
vectors is about $0.6mn$ bytes.  The starting flow $S$ requires little or no memory.
Finally, $F$ and the
associated scratch memory require between $8n^s$ and $16n^2$ bytes.
For a planar graph with $m\approx 3n$, the total should between $13.8n^2$ and $21.8n^2$,
and experiments have shown that $19n^2$ is typical for such planar graphs.

The run time for the all-points shortest path computation depends on sparsity.
The adjacency matrix $A$ has $2m$ non--zeros, and the load factors in Figure~\ref{Flfacs}
sum to 7.5, a number that does not seem likely to grow much as $n$ increases.
Hence, the all-points shortest path computation requires about $15mn$ add-and-minimize
operations.  Each $F\gets S+\omega F$ computation requires one multiply-add for
each non--zero in the redistribution vectors, for a total of about $0.6mn$.
Finally, the $0.6mn$ multiply-add's are repeated $d+1$ times, where the
graph diameter~$d$ is $O(\log n)$.

\noindent {\it Acknowledgement}. The second author was funded by NIST Grant No. 60NANB10D128.

\end{document}